\documentclass[11pt]{article}
\pdfoutput=1

\usepackage{xcolor}
\usepackage{amsmath}
\usepackage{amssymb}
\usepackage{amsthm}
\usepackage{mathtools}
\usepackage{bbm,bm}
\usepackage{hyperref}
\usepackage{theoremref}
\usepackage{enumitem}
\usepackage{stmaryrd} 
\usepackage[colors]{optsys}
\usepackage{eurosym}
\usepackage{graphicx}

\def\reals{\mathbb{R}}
\def\ereals{\overline{\reals}}

\def\comp{\mathop{\text{\scriptsize\raise 1pt \hbox{$\circ$}}}}
\def\infconv{\mathop{\text{\scriptsize\raise 1pt \hbox{$\square$}}}}

\def\argmin{\mathop{\rm argmin}\limits}
\def\argmax{\mathop{\rm argmax}\limits}

\def\minimize{\mathop{\rm minimize}\limits}
\def\maximize{\mathop{\rm maximize}\limits}

\def\st{\mathop{\rm subject\ to}}

\def\dom{\mathop{\rm dom}\nolimits}

\def\ovr{\mathop{\rm over}\ }

\def\Tr{\mathop{\rm Tr}}

\def\upto{{\raise 1pt \hbox{$\scriptstyle \,\nearrow\,$}}}
\def\downto{{\raise 1pt \hbox{$\scriptstyle \,\searrow\,$}}}

\def\epi{\mathop{\rm epi}\nolimits}

\def\FF{(\F_t)_{t=0}^T}

\def\F{{\cal F}}

\def\N{{\cal N}}

\newcommand{\utility}{\upsilon}

\graphicspath{{./},{figures/}}
\usepackage{a4wide}
\usepackage{fullpage}
\newcommand{\nb}[3]{
  {\colorbox{#2}{\bfseries\sffamily\tiny\textcolor{white}{#1}}}
  {\textcolor{#2}{$\blacktriangleright${#3}$\blacktriangleleft$}}}

\newcommand{\mdl}[1]{\nb{Michel}{blue}{#1}}

\title{Optimal Operation and Valuation of Electricity Storages in Intraday Markets \footnote{This
    research benefited from the support of the FMJH Program Gaspard Monge for
    optimization and operations research and their interactions with data
    science.}}

\author{Jean-Philippe Chancelier$^\dagger$ \and
  Michel De Lara\thanks{CERMICS, \'Ecole nationale des ponts et chaussées, IP Paris, France, {jean-philippe.chancelier, michel.delara}@enpc.fr}
  \and
  Fran\c{c}ois Pacaud\thanks{Centre Automatique et Systèmes, Mines Paris-PSL, 60 boulevard Saint Michel, 75006 Paris, francois.pacaud@mineparis.psl.eu}
  \and
  Tanguy Lindegaard\thanks{Gas and Power Origination, Natixis, 7 promenade Germaine Sablon, 75013 Paris, tanguy.lindegaard-ext@natixis.com}
  \and
  Teemu Pennanen\thanks{Department of Mathematics, King's College London, Strand, London, WC2R 2LS, United Kingdom, teemu.pennanen@kcl.ac.uk}
  \and
  Ari-Pekka Perkki\"o\thanks{Mathematics Institute, Ludwig-Maximilian University of Munich, Theresienstr. 39, 80333 Munich, Germany, a.perkkioe@lmu.de}}

\begin{document}
\maketitle

\begin{abstract}
This paper applies computational techniques of convex stochastic optimization to optimal operation and valuation of electricity storages in the face of uncertain electricity prices. Our valuations are based on the indifference pricing principle, which builds on optimal trading strategies and calibrates to the user's financial position, market views and risk preferences. The underlying optimization problem is solved with the Stochastic Dual Dynamic Programming algorithm which is applicable to various specifications of storages, and it allows for e.g.\ hard constraints on storage capacity and charging speed. We illustrate the approach in intraday trading where the agent charges or discharges a battery over a finite number of delivery periods, and the electricity prices are subject to bid-ask spreads and significant uncertainty. Optimal strategies are found in a matter of minutes on a regular PC. We find that the corresponding trading strategies and battery valuations vary consistently with respect to the agent's risk preferences as well as the physical characteristics of the battery.
\end{abstract}

\section{Introduction}

The growing proportion of renewable energy generation increases the
uncertainties and seasonalities of supply and the price of electricity. This creates financial incentives to use batteries and other electricity storages for trading electricity in order to benefit from the price fluctuations. It can be expected that when the overall storage capacity increases and such activities become more popular and the aggregate supply and price fluctuations will diminish. Further discussion and references on the topic can be found e.g.~in \cite{yu2017stochastic} and the survey article~\cite{weitzel2018energy}.

This article proposes mathematical models and computational tools for optimal operation and valuation of batteries and other electricity storages in the context of electricity trading. Economically sensible valuation and optimization techniques will facilitate the expansion of the overall storage capacity which will be necessary for the development of resilient power systems based on renewables. The presented techniques optimize storage management strategies in the face of uncertain electricity prices and they adapt to the trader's financial position, risk preferences as well as different market conditions. Besides the simple model specifications studied in this paper, the techniques can be applied to various specifications of the storage as well as to extensions where the agent might have existing storage, production facilities or consumption or a mixture of the above. 

We start by developing a stochastic optimization model for finding optimal strategies for charging and discharging an energy storage over time subject to physical constraints. The aim is to optimize a user-specified risk measure of the terminal wealth which depends on the uncertain price fluctuations as well as the decision maker's strategy. Realistic models of the storage result in stochastic
optimization problems that fail to be analytically solvable. This is partly due to the fact that a physical storage has capacity constraints which amount to state constraints in the corresponding optimal control formulation. We will employ the {\em stochastic dual dynamic programming} algorithm which is able to handle convex instances of such problems computationally in cases where the underlying risk factors are Markovian; see~\cite{dowson2021sddp} and its references.

Being able to solve the optimal storage management problem numerically allows us
to compute {\em indifference prices} for batteries and other storages. Indifference pricing is a general approach for determining economically meaningful prices of (financial) products whose cashflows cannot be replicated by trading liquid financial instruments; see e.g.\ the collection~\cite{car9} and the references there. In the numerical illustrations, we study an  intraday trading problem where an agent manages a battery by charging or discharging it over 24 delivery periods and the trading takes place in the intraday market. We show, in particular, how the physical characteristics of the storage and the trader's risk preferences affect the valuations.

Optimal management of power systems under uncertainty has been recently surveyed in~\cite{roald2023power}, with the stochastic dual dynamic programming (SDDP) algorithm identified as one of the most promising methods for the management of energy storages~\cite{lohndorf2013optimizing,megel2015stochastic,lw21,zheng2022arbitraging}. As opposed to alternatives such as model predictive control (see e.g.~\cite{arnold2011model}), SDDP takes explicitly into account the future uncertainties when optimizing decisions. Continuous-time approaches based on solving the Hamilton-Jacobi-Bellman equations, on the other hand, struggle with state constraints which are necessary for realistic descriptions of storages which usually involve hard constraints on the amount of stored energy. The recent paper~\cite{pv21} studies optimal operation of a pumped hydroelectric storage with state constraints. The authors assume positive electricity prices and linear utility functions which do not support general risk preferences. They treat state constraints by the ``level set approach" where nondifferentiable penalty functions are added to the objective. Similar techniques were applied by~\cite{gpw23} in the context of mean field games with some applications to renewable energy. Closer to the present paper is~\cite{lw21} which studied gas storage vauation using the SDDP algorithm. As opposed to that of \cite{lw21}, our model allows for lending or borrowing money in a cash account/money market. This avoids ``discounting" future cashflows and it leads to a more natural notion of an indifference price of a storage in terms of money paid at the time of purchase. Moreover, thanks to the explicit modeling of the agent's cash position, we avoid the delicate ``asset-backed trading strategies" of~\cite{lw21} that, at a given point~$t$ in time, involve planned future decisions in addition to the decisions to be implemented at time~$t$.

In the context of electricity markets, indifference pricing has been studied in~\cite{ptw9} and \cite{ccgv} in the continuous-time setting. The first one applied indifference pricing to production facilities that are controlled by switching between a finite number of states. The second one studied the pricing of a general class of structured contracts. A particularly interesting example there is a ``virtual storage" contract, which resembles physical a storage in terms of the generated cashflows. The numerical computations in~\cite{ccgv}, however, are based on finite difference approximations of the associated Hamilton-Jacobi-Bellman (HJB) partial differential equation. As is typical of HJB-based approaches, they do not adapt well to state constraints so they are not ideally appropriate for physical storages.

The rest of this paper is organized as follows. Section~\ref{sec:example} starts with a formulation of the battery management problem as a stochastic optimization problem. The problem is then used to define the indifference price of a battery. The section concludes with a simple model of the electricity prices observed at the end of the intraday market for each delivery period. Section~\ref{sec:convex} reformulation of the storage management problem as an instance of a standard convex control problem. Section~\ref{sec:sddp} gives a short review of the dynamic programming theory and the Stochastic Dual Dynamic Programming algorithm for the general problem formulation from Section~\ref{sec:convex}. Section~\ref{sec:comp} gives the computational results for a battery management problem over 24 delivery periods.

\section{Battery management in the intraday market}\label{sec:example}

Consider the problem of managing a battery by charging or discharging it over a finite number $T$ of delivery periods. Before 30 September 2025, there were 24 one-hour delivery periods in a day while after that date, this was changed to 96 15-minute delivery periods. At the beginning of each delivery period, the trader decides on the amount of energy bought or sold during the delivery period. The energy is used to charge or discharge the battery. The agent aims to maximize the total revenue generated by the electricity trades over $T$~periods.

\S\ref{Imbalance_prices_in_the_intraday_market} gives a brief description of the intraday electricity market and \S\ref{subsec:tom} formulates the optimization problem that will be studied in the remainder of the paper. \S\ref{subsec:idp} recalls the notion  of indifference price and discusses its numerical computation. \S\ref{subsec:price_model} describes the stochastic model that will be used to describe electricity prices in the computational experiments in Section~\ref{sec:comp}.

\subsection{The intraday market}
\label{Imbalance_prices_in_the_intraday_market}

During each delivery period, energy can be bought or sold at the {\em imbalance prices}; see \cite{rte_imbalance_price}. The imbalance prices are published a couple minutes after the end of the delivery period so they remain unknown during the delivery period when the physical electricity trading occurs. 
Electricity can be traded also before the delivery period in the {\em intraday market} which is essentially a futures market for the physical commodities of energy delivered over the delivery periods. Each delivery period is associated with one futures market. Much like exchange-traded securities, the intraday market is based on the {\em double auction} mechanism where traders submit {\em limit orders} in continuous time. All unmatched orders are recoded in the {\em limit order book} that gives the marginal futures prices for buying electricity in a given delivery period.

At any given time $t$, the limit order book for a given delivery period $t'\ge t$ can be represented by a nondecreasing {\em marginal price curve} $s_{t,t'}:\reals\to\reals$ so that the cost of buying $u$ units of energy is given by
\begin{equation}\label{eq:lob}
S_{t,t'}(u):=\int_0^us_{t,t'}(v)dv;
\end{equation}
see e.g.~\cite[Section~2]{pen11}. Since $s_{t,s}$ is nondecreasing, the cost function $S_{t,s}$ is convex, and the best {\em bid} and {\em ask prices} are given by
\begin{equation}\label{eq:bidask}
s_{t,t'}^-=\lim_{h\upto 0}\frac{S_{t,t'}(u)}{u} = \lim_{u\upto 0}s_{t,t'}(u)\quad\text{and}\quad s_{t,t'}^+=\lim_{h\downto 0}\frac{S_{t,t'}(u)}{u} = \lim_{u\downto 0}s_{t,t'}(u);
\end{equation}
see~\cite[Theorem~24.2]{roc70a}. We always have $s^-_{t,t'}<s^+_{t,t'}$, since otherwise, the best bid and ask offers would have been matched and taken out of the limit order book.

\subsection{The optimization problem}
\label{subsec:tom}

When trading at the imbalance prices, one has to make trading decisions before knowing the associated prices. To avoid such a price risk, one can trade in the intraday markets where the futures prices for each delivery period are observable in real time. This gives rise to a continuous-time portfolio optimization problem where one manages a portfolio of futures contracts, one for each delivery period. Such high-dimensional portfolio optimization problems are, in general, difficult to solve. Moreover, a mathematical formulation of such a problem would require a stochastic model for the high-dimensional forward curve that gives the futures prices for each delivery period. Even if one was able to solve the associated optimization problem, one would be faced with the "model risk" that comes with the specification of such high-dimensional stochastic models.  

We take the intermediate approach where energy for the $t$th delivery period is traded in the intraday market but only just before the closing of the corresponding futures market. For example, we could observe the limit order book for the $t$th delivery period a small unit $\epsilon$ of time (e.g.\ 5 seconds) before the intraday market for that period closes. We denote the corresponding cost function by $S_t:=S_{t-\epsilon,t}$; see equation~\eqref{eq:lob}. Unlike the imbalance prices, the cost function $S_t$ is known at time $t$ when making the decision to trade energy in the $t$th delivery period.  

We study a simple storage with a maximum capacity of $\overline X^e$ and maximum charging and discharging speeds of $\overline U$ and $\underline U$, respectively. If $u$ units of energy is input to the storage, the charging level increases by $c(u)$ units. The function $c:\reals\to\reals$ is strictly increasing and satisfies $c(u)\le u$ for every $u\in\reals$. Again, negative $u$ corresponds to discharging the battery. If $x$~units of energy is stored at the beginning of period $[t,t+1[$, the charging level will be $(1-l)x$ units at the end of the period. The constant $l\in[0,1]$ describes proportional charge loss per period.

To formulate the battery management problem mathematically, we describe the uncertainties and information by a probability space $(\Omega,\F,P)$ and a filtration $\FF$, respectively, and assume that each cost function $S_t$ is an $\F_t$-measurable convex {\em normal integrand}\footnote{Recall that a function $f:\reals^n\times\Omega\to\ereals$ is a {\em normal integrand} if the epigraphical mapping $\omega\mapsto\epi f(\cdot,\omega)$ is a closed-valued and measurable; see e.g.~\cite[Chapter~14]{rw98}.}. The problem can then be written as
\begin{equation}\label{sps}
  \begin{aligned}
    &\maximize\ &  \EE[&v(X^m_T)]\quad \ovr\ (X^m,X^e,U)\in\N\\
    &\st & X^m_0&=x_0^m,\\
    & & X^e_0&=0,\\
    & & X^m_t&= X^m_{t-1}-S_t(U_t)\quad t=1,\ldots,T,\\
    & & X^e_t&= (1-l) X^e_{t-1} + c(U_t)\quad t=1,\ldots,T,\\
    & & X^e_t&\in[0,\overline X^e]\quad t=1,\ldots,T,\\
    & & U_t&\in[-\underline U,\overline U]\quad t=1,\ldots,T,
  \end{aligned}
\end{equation}
where $X^m_t$ is the cash position and $X^e_t$ is the amount of stored energy, respectively, at time~$t$. The control variable $U_t$ specifies the amount of energy bought at time~$t$. Here and in what follows, $\N$ denotes the space of {\em nonanticipative} (i.e.\ $\FF$-adapted) processes. The constraint $(X^m,X^e,U)\in\N$ thus means that decisions taken at time $t$ can only depend on the information observed by time $t$. As usual, negative purchases are interpreted as sales. The initial wealth $x^m_0$ is fixed. The {\em utility function} $\utility \colon \reals\to\reals$ is assumed increasing and concave.

\subsection{Indifference pricing}\label{subsec:idp}

Indifference pricing is a general principle for defining economically meaningful values of (financial) products in {\em incomplete markets} where prices cannot be determined by arbitrage arguments alone; see e.g.~\cite{car9,pen14} and their references. In general, indifference prices are {\em subjective} in the sense that they depend on an agent's existing financial position, risk preferences and views concerning the future values of the underlying risk factors (as described by their probabilistic model). For replicable claims, however, the indifference prices coincide with risk-neutral valuations; see~\cite[Theorem~4.1]{pen14}.

We will study an agent who considers renting a battery for a fixed time period $[0,T]$. The {\em indifference price} for renting a battery is the maximum amount the agent could pay for it without worsening their financial position. In the context of problem~\eqref{sps}, the indifference price is given by
\begin{equation}\label{eq:idp}
\pi := \sup\bset{p\in\reals}{\varphi(x^m_0-p,\overline X^e)\ge\varphi(x^m_0,0)},
\end{equation}
where, for any $x,X\in\reals$, $\varphi(x,X)$ denotes the optimum value of~\eqref{sps} when $x_0^m=x$ and $\overline X^e=X$. Thus, $\varphi(x_0^m-p,\overline X^e)$ is the highest achievable expected utility when the agent pays $p$ units of cash in exchange for getting access to a battery with maximum capacity $\overline X^e$. The value $\varphi(x_0^m,0)$, on the other hand, is the highest achievable expected utility for the agent if they decide not to rent the battery. The indifference price is thus the highest sensible price for an agent who measures their wellbeing by the expected utility of terminal wealth as in problem~\eqref{sps}. The definition also assumes that the agent is able to optimize their decision strategies with and without the battery.

Note that, in this simple model where the agent either manages the battery or does nothing at all, we simply have $\varphi(x_0^m,0)=v(x_0^m)$. The indifference principle can equally well be applied in situations where the agent already owns storage and is considering renting more of it. Similarly, one could study situations where the agent has consumption or production facilities or any mixture of the above. In model~\eqref{sps}, we have assumed that all revenue is invested in cash, but one could extend the model also by allowing investments in alternative assets. The model would then involve additional decision variables that describe the investment strategy. The indifference pricing principle applies equally well in such more complicated situations.

If the function $x\mapsto\varphi(x,U)$ is continuous and strictly increasing, the indifference price $\pi$ is the unique solution to the equation
\begin{equation}\label{eq:root}
\varphi(x^m_0-\pi,\overline X^e)=\varphi(x^m_0,0)
\end{equation}
provided a solution exists. Finding a solution of the equation~\eqref{eq:root} is a nontrivial problem in general since a single evaluation of the function $\varphi$ requires the numerical solution of problem~\eqref{sps}. The task simplifies, however, in the case of the exponential utility
function
\begin{equation}\label{eq:exp}
\utility(z) = \frac{1}{\rho}[1-\exp(-\rho z)],
\end{equation}
where $\rho>0$ is a given risk aversion parameter. 

\begin{proposition}\label{prop:idp}
In the case of the exponential utility function~\eqref{eq:exp}, the indifference price~\eqref{eq:idp} is given by
\[
\pi = -\frac{\ln(1-\rho\varphi(0,\overline X^e))}{\rho}.
\]
\end{proposition}

\begin{proof}
The change of variables $X_t^m\to X_t^m-x$ in problem~\eqref{sps} shows that the optimum value $\varphi(x^m_0,\overline X^e)$ is not affected if we set the initial wealth $x_0^m$ to zero and add $x_0^m$ units of cash to the terminal wealth $X_T^m$ instead. Thus, in the case of exponential utility, the optimum value function satisfies
\[
\varphi(x,U)=e^{-\rho x}\varphi(0,U) +\frac{1}{\rho}(1-e^{-\rho x}).
\]
Equation~\eqref{eq:root} can thus be written as
\[
e^{-\rho(x_0^m-\pi)}\varphi(0,\overline X^e) + \frac{1}{\rho}(1-e^{-\rho(x_0^m-\pi)}) = e^{-\rho x_0^m}\varphi(0,0) +\frac{1}{\rho}(1-e^{-\rho x_0^m})
\]
or equivalently
\[
e^{\rho \pi}\varphi(0,\overline X^e) - \frac{1}{\rho}e^{\rho\pi} = \varphi(0,0) - \frac{1}{\rho}.
\]
This has the unique solution
\[
\pi = \frac{\ln(1-\rho\varphi(0,0))-\ln(1-\rho\varphi(0,\overline X^e)) }{\rho},
\]
where $\varphi(0,0)=0$.
\end{proof}

Proposition~\ref{prop:idp} implies that the computation of the indifference price of the battery in~\eqref{eq:idp} requires the solution of only one instance of problem~\eqref{sps}. In problems where the agent has e.g.\ production facilities, the computation of the indifference price would require the solution of two optimization problems since, in that case, the computation of the optimum value $\varphi(0,0)$ obtained without renting the battery becomes nontrivial, unlike in the proof above.

\subsection{The electricity price model}\label{subsec:price_model}

In the numerical illustrations in Section~\ref{sec:comp} below, we assume both $S_t$ and $c$ to be positively homogeneous in the decision variables. More specifically, we will assume that
\begin{equation}\label{eq:ssl}
S_t(u,\omega)=
\begin{cases}
    s^+_t(\omega)u & \text{if $u\ge 0$},\\
    s^-_t(\omega)u & \text{if $u\le 0$},
\end{cases}
\end{equation}
where $s^+_t$ and $s^-_t$ denote the best ask- and the best bid-prices, respectively, just before the $t$th delivery period; see~\eqref{eq:bidask}. The bid and ask prices are modeled as stochastic processes. The above assumption can be justified for a small trader whose trades are unlikely to exceed the liquidity available at the best bid and ask prices. Larger orders may dig deeper into the limit order book so that the marginal prices would move against the trader according to equation~\eqref{eq:lob}. Simple stochastic models for the full limit order book can be found e.g.\ in~\cite{mp12}.

For simplicity, we will assume that $s^-_t=s_t-\delta$ and $s^+_t=s_t+\delta$, where $\delta>0$ is constant and only the {\em mid-price} $s_t$ is stochastic. We will assume that
\[
s_t = \bar s_t+\xi_t,
\]
where $\xi_t$ is a zero-mean stochastic process and $\bar s_t$ is the {\em day-ahead} price for the $t$th delivery period observed after the day-ahead auction. The logic is that, being a futures price for energy, the day-ahead price should be close to the expectation of the future intraday prices. Consistent positive or negative deviations from day-ahead prices would allow for simple ``statistical arbitrage" trades, so one would expect the prices to converge according to supply and demand. Figure~\ref{fig:time_series} plots the historical time series of day-ahead and ID1 prices for France in 2024 at hourly granularity. The ID1 index is a weighted average of executed trades during the last hour before the closure of the intraday market. We take it as a proxy for the intraday mid-price in our statistical model below.

We will describe $\xi=(\xi_t)$ by the autoregressive time series model
\begin{equation}\label{eq:ar_1}
    \xi_t = a\xi_{t-1} +\epsilon_t
\end{equation}
where $\epsilon_t$ are iid Gaussians with zero mean.
\begin{figure}[!ht]
  \centering  \includegraphics[width=.7\textwidth]{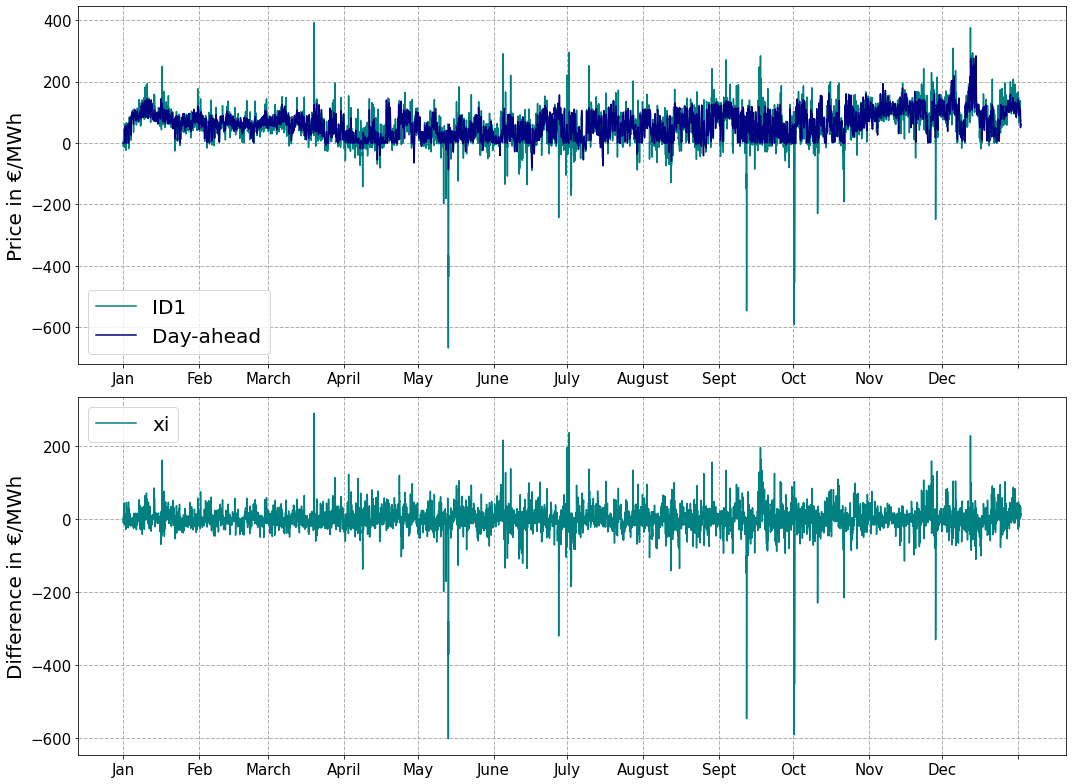}
  \caption{Time series of day-ahead, ID1 prices and $\xi$ for France in 2024 \label{fig:time_series}}
\end{figure}
We estimate the coefficient $a\in\reals$ with a linear regression from $8760$ observations of $\xi_t$ (365 days of hourly data). Table \ref{tab:linear_reg} sums up the linear regression computations. Given the coefficient estimate $\hat a$, we 
assume that the random term $\epsilon_t$ in~\eqref{eq:ar_1} follows the zero-mean Gaussian distribution fitted to the residuals
\begin{equation}
    r_t = \xi_{t+1}-\hat a\xi_t.
\end{equation}

\begin{table}[h]
    \centering
    \begin{tabular}{c|c|c}
         & Value & p-value \\
         \hline
        Slope $\hat a$ & 0.48 & 0.00 \\
        \hline
        Intercept & -0.13 & 0.65 \\
        \hline
        $R^2$ & 0.23 & .\\ 
    \end{tabular}
    \caption{Linear regression parameters}
    \label{tab:linear_reg}
\end{table}



\section{Convex control reformulations}\label{sec:convex}

We will formulate the storage management problem as a {\em convex stochastic optimal control} problem. For generality, we will allow for the general control format of~\cite[Section~6.2]{ccdmt} where, at each stage $t$, the system may be controlled by both $\F_{t-1}$-measurable and $\F_t$-measurable controls; see
also~\cite{pp25a}. Such formulations arise naturally, e.g.\ when trading in both intraday and imbalance markets. Other instances can be found
in~\cite{pdccjr}.

Consider the optimal control problem
\begin{subequations}
  \label{oc}
  \begin{align}
    &\minimize & \EE\Bigg[\sum_{t=0}^{T-1}&  L_t(X_t,U^b_t,U^a_{t+1})+ J(X_T)\Bigg]\ \ovr\ (X,U^b,U^a)\in\N,\\ 
    &\st\ & X_0 &= x_0,\\
    & & X_{t+1} &= F_{t}(X_t,U^b_t,U^a_{t+1})\quad t=0,\dots,T-1\ a.s.,
                  \label{oc_c}
\end{align}
where the state $X_t$ and the controls\footnote{The superscript~$b$ in~$U^b_t$ stands for ``before'' (in the sense of a decision taken at the beginning of the time interval~$[t,t+1[$), whereas the superscript~$a$ in~$U^a_t$ stands for ``after'' (in the sense of a decision taken at the end of the time interval~$[t,t+1[$).} $U^b_t$ and $U^a_t$ are random variables taking values in $\reals^N$, $\reals^{M^b}$ and $\reals^{M^a}$. The functions $L_t:\reals^N\times\reals^{M^b}\times\reals^{M^a}\times\Omega\to\ereals$ and $J:\reals^N\times\Omega\to\ereals$ are extended real-valued convex normal integrands and, for each $t$,
\begin{equation*}
  \label{eq:system_equations}
  F_{t}(x_t,U^b_t,U^a_{t+1},\omega) := A_{t+1}(\omega)x_t+B^b_{t+1}(\omega)U^b_t+B^a_{t+1}(\omega)U^a_{t+1}+W_{t+1}(\omega),
\end{equation*}
\end{subequations}
where $A_{t+1}$, $B^b_{t+1}$ and $B^a_{t+1}$ are $\F_{t+1}$-measurable random matrices of appropriate dimensions and $W_{t+1}$ are $\F_{t+1}$-measurable random vectors. As before, the set $\N$ denotes the space of adapted stochastic processes. The vector $x_0\in\reals^N$ is a given initial state. 

The format of problem~\eqref{oc} is quite general and can accommodate various kinds of storages and market specifications. Besides storage management and electricity trading, it can include production facilities and/or consumption of
electricity. The computational techniques discussed below apply to such
situations as well. Both production and consumption may create stronger
incentives to buy storages. The indifference pricing approach described in
Section~\ref{subsec:idp} applies directly to such modifications.

Problem~\eqref{sps} does not fit the format of~\eqref{oc}, in general, since the functions $S_t$ and $c$ are nonlinear. However, when the cost functions are of the form~\eqref{eq:ssl}, we can rewrite the budget constraint in~\eqref{sps} as the linear equation
\[
X^m_t=X^m_{t-1} - s^+_tU^+_t + s^-_tU^-_t,
\]
where $U^+_t$ and $U^-_t$ are the positive and negative parts, respectively, of the control variable~$U_t$. As to the battery, we will assume that the function $c$ is of the form 
\begin{equation}
  \label{eq:function_c}
c(u)=
\begin{cases}
    c^+u & \text{if $u\ge 0$},\\
    c^-u & \text{if $u\le 0$}
\end{cases}  
\end{equation}
where $c^+$ and $c^-$ are constants such that $0<c^+\le c^-$. We can then write the battery constraint in the linear form
\[
X^e_t=(1-l)X^e_{t-1}+c^+U^+_t-c^-U^-_t.
\]
When the cost function is given by Equation~\eqref{eq:ssl} and the function $c$ is given by Equation~\eqref{eq:function_c},  Problem~\eqref{sps} can be written as
\begin{subequations}
  \label{spsl0}
  \begin{align}
    \maximize\quad &  \EE[v(X^m_T)]\quad \ovr\ (X^m,X^e,U)\in\N \nonumber \\
    \st\quad & X^m_0=x_0^m,\nonumber\\
            & X^e_0=0,\nonumber \\
            & X^m_t= X^m_{t-1}- s^+_tU^+_t + s^-_tU^-_t\quad t=1,\ldots,T, \label{eq:spsl0-Xm}\\
            & X^e_t= (1-l)X^e_{t-1} +c^+U^+_t-c^-U^-_t\quad t=1,\ldots,T,\label{eq:spsl0-Xe}\\
            & X^e_t\in[0,\overline X^e]\quad t=0,\ldots,T, \label{eq:spsl0-Xe-bounds}\\
            & U^+_t\in[0,\overline U] \quad t=0,\ldots,T, \label{eq:spsl0-U-plus}\\
               & U^-_t\in[0,\underline U] \quad t=0,\ldots,T
                 \eqfinp \label{eq:spsl0-U-minus}
  \end{align}
\end{subequations}
Consider also the relaxation
\begin{subequations}
  \label{spsl}
  \begin{align}
    \maximize\quad &  \EE[v(X^m_T)]\quad \ovr\ (X^m,X^e,\hat U,\check U)\in\N \nonumber\\
    \st\quad & X^m_0=x_0^m,\nonumber\\
     & X^e_0=0,\nonumber\\
     & X^m_t= X^m_{t-1}- s^+_t\hat U_t + s^-_t\check U_t\quad t=1,\ldots,T,\label{eq:spsl-Xm}\\
     & X^e_t= (1-l)X^e_{t-1} +c^+\hat U_t-c^-\check U_t\quad t=1,\ldots,T,\\
     & X^e_t\in[0,\overline X^e]\quad t=0,\ldots,T,\label{eq:spsl-Xe}\\
     & \hat U_t\in[0,\overline U] \quad t=0,\ldots,T, \label{eq:spsl-hat-U} \\
     & \check U_t\in[0,\underline U] \quad t=0,\ldots,T,
  \end{align}
\end{subequations}
where we allow $\hat U$ and $\check U$ to take strictly positive values at the
same time.  Problem~\ref{spsl} is a {\em convex} stochastic control problem. If
we add the complementarity constraint that $\hat U_t\check U_t=0$ for every~$t$,
then Problem~\eqref{spsl} is equivalent to Problem~\eqref{spsl0}. Thus, Problem~\eqref{spsl} is a relaxation of Problem~\eqref{spsl0} and thus, its optimum value of~\eqref{spsl} is at least as high as that of~\eqref{spsl0}. Adding the complementarity constraint would, however, destroy
convexity and make the problem difficult to solve, in general.  The following
Proposition~\ref{pr:complementarity_condition} gives a simple condition under
which the complementarity condition is redundant in the sense that, the optimum values in problems \eqref{spsl} and \eqref{spsl0} coincide and the optimal solutions of problem \eqref{spsl0} are easily constructed from the solutions of \eqref{spsl}.

\begin{proposition}\label{pr:complementarity_condition}
Assume that
\begin{equation}
  \frac{s_t^-}{c^-}\le \frac{s_t^+}{c^+}.
  \label{eq:complementarity_condition_assumption}
\end{equation}
Then the optimum values of~\eqref{spsl0} and~\eqref{spsl} coincide. Moreover, if $(X^m,X^e,\hat U,\check U)$ is feasible in~\eqref{spsl}, and we define $(\bar X^m,\bar X^e,\bar U)$ by setting 
\begin{equation}\label{eq:Ct-one}
C_t := c^+\hat U_t-c^-\check U_t
\end{equation} 
and
\begin{subequations}
\begin{align}
\bar U_t&:=C_t^+/c^+-C_t^-/c^-,
\label{eq:complementarity_condition_definition_a} 
\\ 
\bar X^m_t&:= \bar X^m_{t-1}-s^+_t\bar U_t^++s_t^-U_t^-, \quad \bar X^m_0=x^m_0,
\label{eq:complementarity_condition_definition_b} \\
\bar X^e_t&:=X^e_t, \quad \bar X^e_0=0,
\end{align}
\end{subequations}
then $(\bar X,\bar U)$ is feasible in~\eqref{spsl0},
\begin{equation}\label{objectives}
\EE[v(\bar X^m_T)] \ge \EE[v(X^m_T)].
\end{equation}
\end{proposition}

\begin{proof} As already pointed out the value of Problem~\eqref{spsl0} is smaller that the value of Problem~\eqref{spsl}.
  It remains to prove the converse.
  
  For that purpose, consider $(X^m, X^e,\hat U,\check U)$ feasible in Problem~\eqref{spsl}, and define $\bar U_t$,
  for all $t=1,\ldots,T$, as follows
  \begin{equation}\label{eq:Ct}
    \bar U_t :=C_t^+/c^+-C_t^-/c^- \text{ with } C_t := c^+\hat U_t-c^-\check U_t
    \eqfinp
  \end{equation}
  First, we prove that $(\bar X^m, \bar X^e, \bar U)$, where $\bar X^m$ and $\bar X^e$
  are given by Equations~\eqref{eq:spsl0-Xm} and~\eqref{eq:spsl0-Xe} is feasible in
  in Problem~\eqref{spsl0}.
  For each $t=1,\ldots,T$, we have to consider two cases according to the sign of $C_t$ but we just consider the case
  $C_t \ge 0$ as the case $C_t \le 0$ is treated in a similar way.
  If $C_t \ge 0$, then by~\eqref{eq:Ct} we have that $\bar U_t\ge 0$ and we obtain that
  \begin{align}
    \bar U^-_t =0 \text{ and }  0 \le \bar U^+_t = \bar U_t
    &= \hat U_t-\frac{c^-}{c^+}\check U_t\le \hat U_t\le \overline U \eqfinv
      \label{eq:U-constraint}
  \end{align}
  using~\eqref{eq:spsl-hat-U} in the last inequality. We obtain from~\eqref{eq:U-constraint}
  that $\bar U$ satisfy the constraints~\eqref{eq:spsl0-U-plus} and~\eqref{eq:spsl0-U-minus}.
  Moreover, if $C_t \ge 0$, we also have 
  \begin{align}
    s^+ \bar U^+_t- s^- \bar U^-_t
    &= s^+ \bar U_t \tag{as $\bar U^+_t = \bar U_t$ and $\bar U^-_t = 0$ when $\bar U_t\ge 0$}
    \\
    &= s^+ \bp{\hat U_t - \frac{c^-}{c^+} \check U_t}
      \tag{by~\eqref{eq:Ct}}
    \\
    &\le s^+ \hat U_t - s^- \check U_t
      \label{eq:cost-ineq}
      \intertext{using assumption~\eqref{eq:complementarity_condition_assumption} in the last inequality, and also}
    c^+ \bar U^+_t- c^- \bar U^-_t
    &= c^+ (C_t/c^+) \tag{by~\eqref{eq:Ct} as $C_t^+ = C_t$ and $C_t^-=0$ when $C_t \ge 0$} \\
    &= c^+_t\hat U_t - c^-_t\check U_t
      \label{eq:controls-equal}
      \eqfinv 
  \end{align}
  using again the definition of $C_t$ in~\eqref{eq:Ct} in the last equality.
  Proceeding in a similar way when $C_t \le 0$ also gives~\eqref{eq:cost-ineq} and \eqref{eq:controls-equal}.

  Now, using Equation~\eqref{eq:controls-equal} and initial conditions, it is
  immediate to obtain that $X^e$ and $\bar X^e$ coincide and therefore
  $\bar X^e$ satisfy constraint~\eqref{eq:spsl0-Xe-bounds} as $X^e$ satisfy
  constraint~\eqref{eq:spsl-Xe}.  Thus, we have proved that
  $(\bar X^m, \bar X^e, \bar U)$, is feasible in Problem~\eqref{spsl0}.

  Second we prove that
  \begin{equation}
    \label{eq:cost-ineq1}
    \EE[v(X^m_T)] \le \EE[v(\bar X^m_T)]
    \eqfinp
  \end{equation}
  As the utility function~$v$ is increasing, it is enough to prove that $X^m_T \le \bar X^m_T$.
  We prove this result by induction on $t$. First we have that $ X^m_0 = x^m_0 = \bar X^m_0$.
  Second, assume that $X^m_{t-1} \le \bar X^m_{t-1}$ for $0 \le t-1 < T$. Then, we have
  \begin{align}
    X^m_{t} &= X^m_{t-1}- \bp{s^+_t\hat U_t - s^-_t\check U_t} \tag{by~\eqref{eq:spsl-Xm}} \\
            &\le \bar X^m_{t-1} - \bp{s^+_t\hat U_t - s^-_t\check U_t} \tag{by induction assumption} \\
            &\le \bar X^m_{t-1} - \bp{ s^+ \bar U^+_t- s^- \bar U^-_t}\tag{by~\eqref{eq:cost-ineq}} \\
            &= \bar X^m_t \tag{by~\eqref{eq:spsl0-Xm}}
              \eqfinv
  \end{align}
  and the result follows.

  Finally, given an admissible solution $(X^m, X^e,\hat U,\check U)$ feasible in Problem~\eqref{spsl},
  we have obtained a feasible solution in Problem~\eqref{spsl0} satisfying Equation~\eqref{eq:cost-ineq1}.
  Thus, the cost associated to $(X^m, X^e,\hat U,\check U)$ in Problem~\eqref{spsl} is
  smaller than the value of  $(\bar X^m, \bar X^e, \bar U)$ in Problem~\eqref{spsl0} and thus smaller that
  ${\cal V}(\eqref{spsl0})$ and we conclude that ${\cal V}(\eqref{spsl}) \le {\cal V}(\eqref{spsl0})$,
  which ends the proof.
\end{proof}

Problem~\eqref{spsl} is in the format of~\eqref{oc} without the here-and-now controls $U^b$. Indeed, problem~\eqref{sps} is then an instance of~\eqref{oc} with $N=2$, $M^b=0$, $M^a=2$,
\[
J(x_T)=-\utility(x_T^m),
\]
$U^a_t=(\check U_t,\hat U_t)$ and
\[
L_t(X_t,U^b_t,U^a_{t+1}) = 
\begin{cases}
    0 & \text{if $X^e_t\in[0,\overline X^e]$, $\hat U_{t+1}\in[0,\overline U]$ and $\check U_{t+1}\in[0,\underline U]$},\\
    +\infty & \text{otherwise}.
\end{cases}
\]
We can thus use the numerical techniques described in Section~\ref{subsec:dmp}
and Section~\ref{subsec:sddp} below. In particular, since there are no here-and-now controls, we can use the SDDP Julia package described
in~\cite{dowson2021sddp}.



\section{Stochastic dual dynamic programming}
\label{sec:sddp}

In general, problem~\eqref{oc} does not allow for analytical solutions, so we will resort to computational techniques based on dynamic programming. \S\ref{Stochastic_dynamic_programming} reviews the general dynamic programming theory for optimal control problems of the form~\eqref{oc}. In particular, Theorem~\ref{opoc} characterizes optimal solutions in terms of Bellman functions (aka cost-to-go functions) while Theorem~\ref{thm:markov} shows that when the underlying randomness is driven by a Markov process, the Bellman functions only depend on the current system state~$X_t$ and the current value of the Markov process.
\S\ref{subsec:dmp} presents the discretization of Markov processes.
\S\ref{subsec:sddp} describes a version of the {\em stochastic dual dynamic
programming} algorithm for problems of the form~\eqref{oc}.

\subsection{Stochastic dynamic programming}
\label{Stochastic_dynamic_programming}

We will assume throughout that there is a stochastic process
$\xi=(\xi_t)_{t=0}^T$ such that $L_{t-1}$, $ A_t$, $ B^b_t, B^a_t$ and $ W_t$
in~\eqref{eq:system_equations} depend on $\omega$ only through
$\xi^t:=(\xi_s)_{s=0}^t$ and $J_T$ depends on $\omega$ only through $\xi^T$. We will also
assume that $\xi_t$ takes values in $\reals^{d_t}$ and we denote
\[
d^t:=d_0+\cdots+d_t
\]
and $d=d^T$. Without loss of generality, we will assume that $\Omega=\reals^d$ and
that $\xi_t(\omega)$ is the projection of $\omega$ to its $t$th component and that $\F_t$ is the $\sigma$-algebra on $\reals^d$ generated by the projection
$\xi\mapsto\xi^t:=(\xi_0,\ldots,\xi_t)$. In particular, $P$ is the distribution of~$\xi$. With a slight misuse of notation we will write 
$L_t(x_t,U^b_t,U^a_{t+1},\xi^{t+1})$ instead of $L_t(x_t,U^b_t,U^a_{t+1},\omega)$ and similarly for other $\F_{t+1}$-measurable random elements. 

It follows that the $\F_t$-conditional expectation of an $\F_{t+1}$-measurable quasi-integrable random variable $\xi\mapsto\eta(\xi^{t+1})$ has a pointwise representation
\[
\EE\big[\eta(\xi^{t+1})| \xi^t\big] := \int_{\reals^{d_{t+1}}}\eta(\xi^t,\xi_{t+1})d\nu_t(\xi^t,d\xi_{t+1})\quad\forall \xi^t\in\reals^{d^t},
\]
where $\nu_t$ is a {\em regular $\xi^t$-conditional distribution of $\xi_{t+1}$}; see, e.g.,~\cite[Section~2.1.5]{pp24}. Note that $\xi^{t+1}=(\xi^t,\xi_{t+1})$. We say that a function $h:\reals^n\times\reals^{d^{t+1}}\to\ereals$ is {\em lower bounded} if there is an integrable function $m:\reals^{d^{t+1}}\to\reals$ such
that $h(x,\xi^{t+1})\ge m(\xi^{t+1})$ for all
$(x,\xi^{t+1})\in\reals^n\times\reals^{d^{t+1}}$. Given a lower bounded measurable
function $h:\reals^n\times\reals^{d^{t+1}}\to\ereals$, we denote
\[
\EE[h(x,\xi^{t+1})|\xi^t] := \int_{\Xi_{t+1}}h(x,\xi^t,\xi_{t+1})d\nu_t(\xi^t,d\xi_{t+1})\quad\forall (x,\xi^t)\in\reals^n\times\reals^{d^t}.
\]
  
We say that a sequence of lower bounded measurable functions
$J_t:\reals^N\times\reals^{d^t}\to\ereals$ solves the Bellman equations for~\eqref{oc}
if $J_T=J$ and 
\begin{equation}\label{be}
  J_t(x_t,\xi^t) = \inf_{U^b_t}\left\{\EE\left[\inf_{U^a_{t+1}}\{L_t(x_t,U^b_t,U^a_{t+1},\xi^{t+1}) + J_{t+1}(F_{t}(x_t,U^b_t,U^a_{t+1},\xi^{t+1}),\xi^{t+1})\}\Bigg|\xi^t\right]\right\}
\end{equation}  
for all $t=0,\ldots,T-1$ and $P$-almost every $\xi$.
Bellman equations of the above form~\eqref{be}, together with optimal feedbacks, can be derived\footnote{%
  In \cite[Sect.~6]{ccdmt}, Proposition~13 can be extended to the case where instantaneous costs, dynamics
  and conditional distribution of the next uncertainty all three depend upon the whole sequence of past uncertainties (and not only upon the last uncertainty).
  The proof requires an adaptation of Appendix~D in the preprint version of~\cite{ccdmt} (which contains additional
material compared to the published version~\cite{ccdmt}). 
We can then obtain a variant of \cite[Equation~(42)]{ccdmt} that corresponds to Equation~\eqref{be} above.
Then, optimal feedbacks can be deduced from \cite[Subsect.~3.3]{ccdmt}.
}
from the framework of~\cite[Sect.~6]{ccdmt} (and \cite[Subsect.~3.3]{ccdmt} for optimal feedbacks), 
which first analyzed 
Bellman equations in the case of nonlinear system dynamics
with decision-hazard-decision information structures.
The equations are well-motivated by the following optimality conditions from \cite[Theorem~1]{pp25a}.

\begin{theorem}[Optimality principle]\thlabel{opoc}
If $(J_t)_{t=0}^T$ is a solution of~\eqref{be} such that each
$J_t:\reals^N\times\Xi^t\to\ereals$ is a lower bounded convex normal integrand, then the optimum value of~\eqref{oc} equals $J_0(x_0)$ and an $(\bar X,\bar U^b,\bar U^a)\in\N$ solves~\eqref{oc} if and only if $\bar X_0=x_0$ and
\begin{align*}
\bar U^b_t(\xi^t)&\in \argmin_{U^b_t\in\reals^{M^b}}
      \left\{\EE\left[\inf_{U^a_{t+1}\in\reals^{M^a}}\{L_t(\bar X_t,U^b_t,U^a_{t+1},\xi^{t+1}) + J_{t+1}(F_{t}(\bar X_t,U^b_t,U^a_{t+1},\xi^{t+1}),\xi^{t+1})\}\Bigg|\xi^t\right]
      \right\}\eqfinv
    \\
    \bar U^a_{t+1}(\xi^{t+1})
    &\in \argmin_{U^a_{t+1}\in\reals^{M^a}}\{L_t(\bar X_t(\xi^t),\bar U^b_t(\xi^t),U^a_{t+1},\xi^{t+1})
     + J_{t+1}(F_{t}(\bar X_t(\xi^t),\bar U^b_t(\xi^t),U^a_{t+1},\xi^{t+1}),\xi^{t+1})\},
    \\
    \bar X_{t+1}(\xi^{t+1})
    &= F_{t}(\bar X_t(\xi^t),\bar U^b_t(\xi^t),\bar U^a_{t+1}(\xi^{t+1}),\xi^{t+1})
  \end{align*}
  for all $t=0,\ldots,T-1$ and $P$-almost every $\xi$.
\end{theorem}

The following gives sufficient conditions for the existence of solutions to the
Bellman equations~\eqref{be}. Given a normal integrand
$h:\reals^n\times\Omega\to\ereals$, the function~$h^\infty$, obtained from $h$ by defining
$h^\infty(\cdot,\omega)$, for each $\omega\in\Omega$, as the {\em recession function} of
$h(\cdot,\omega)$, is a positively homogeneous convex normal integrand; see
e.g.~\cite[Example~14.54a]{pp24} or~\cite[Theorem~1.39]{pp24}. The following
is~\cite[Theorem~2]{pp25a} with the first stage running cost redefined as
$+\infty$ unless $X_0=x_0$.

\begin{theorem}[Existence of solutions]\label{existoc}
  Assume that the set
  \begin{multline*}
    \Big\{(X,U^b,U^a)\in\N \, \Big\vert \,
    \sum_{t=0}^{T-1}  L^\infty_t(X_t,U^b_t,U^a_{t+1})+ J^\infty(X_T)\le 0,\\
    X_0=0\eqsepv X_t= A_tX_{t-1} +  B^b_tU^b_{t-1} + B^a_tU^a_t\eqsepv t=1,\dots,T\ a.s.
    \Big\}
  \end{multline*}
  is linear. Then the control problem~\eqref{oc} has an optimal solution and the
  associated Bellman equations~\eqref{be} have a unique solution $(J_t)_{t=0}^T$
  of lower-bounded convex normal integrands.
\end{theorem}

The following is from~\cite[Theorem~3]{pp25a}.

\begin{theorem}[Markovianity]\label{thm:markov}
  Let $(J_t)_{t=0}^T$ be a sequence of lower bounded normal integrands that
  solve~\eqref{be}. Assume that $\xi=(\xi_t)_{t=0}^T$ is a Markov process and that
  $L_{t-1}$, $A_t$, $B^b_t$, $B^a_t$ and $W_t$ depend on $\xi$ only through
  $(\xi^t,\xi^{t-1})$. Then $J_t$ depends on $\xi^t$ only through $\xi_t$. If
  $\xi$ is a sequence of independent random variables and if $L_{t-1}$, $A_t$,
  $B^b_t$, $B^a_t$ and $W_t$ do not depend on $\xi^{t-1}$, then $J_t$ do not
  depend on $\xi$ at all.
\end{theorem}

\subsection{Discretization of Markov processes}\label{subsec:dmp}

In general, the Bellman equations~\eqref{be} do not allow for analytical
solutions. We will therefore resort to numerical approximations and assume, as
in \autoref{thm:markov}, that there is a Markov process $\xi=(\xi_t)_{t=0}^T$ such
that $L_{t-1}$, $A_t$, $B_t$ and $W_t$ depend on $\xi$ only through
$(\xi_{t-1},\xi_t)$. 
We will first approximate the Markov process~$\xi$ by a finite-state Markov chain,
and then use the Stochastic Dual Dynamic Programming (SDDP) algorithm to solve
the discretized problem; see \S\ref{subsec:sddp} below.

We will discretize $\xi$ using integration quadratures as
in~\cite{tauchen1991quadrature} in the context of option pricing. To that end,
we assume that the $\xi_t$-conditional distribution of $\xi_{t+1}$ has density
$p_t(\cdot|\xi_t)$. Given a strictly positive probability density $\phi_{t+1}$ on
$\reals^{d_{t+1}}$, the conditional expectation of a random variable
$\xi\mapsto\eta(\xi^{t+1})$ will be approximated by
\begin{align*}
\EE[\eta |\xi_t] &= \int_{\reals^{d_{t+1}}}\eta(\xi_{t+1}) p_t(\xi_{t+1}|\xi_t)d\xi_{t+1} \\
&=\int_{\reals^{d_{t+1}}}\eta(\xi_{t+1})\frac{p_t(\xi_{t+1}|\xi_t)}{\phi_{t+1}(\xi_{t+1})}\phi_{t+1}(\xi_{t+1})d\xi_{t+1}\\
&\approx \sum_{i=1}^N\eta(\xi_{t+1}^i)\frac{p_t(\xi_{t+1}^i|\xi_t)}{\phi_{t+1}(\xi_{t+1}^i)}w_{t+1}^i,
\end{align*}
where $(\xi_{t+1}^i,w_{t+1}^i)_{i=1}^N$ is a quadrature approximation of the
density $\phi_t$. A natural choice for $\phi_{t+1}$ would be the unconditional density
of $\xi_{t+1}$ provided it is available in closed form. Constructing such
approximations for $t=0,\ldots,T$ and conditioning each $\xi_{t+1}$ with the quadrature
points $\xi_t^i$ obtained at time $t$ results in a finite state Markov chain
where, at each~$t$, the process~$(\xi_t)_{t=0}^T$ takes values in the finite set
$\{\xi_t^j\}_{j=1}^N$ and the transition probability from $\xi_t^j$ to
$\xi_{t+1}^i$ is given by
\[
p_t^{ji}:=\frac{p_t(\xi_{t+1}^i|\xi^j_t)}{\phi_{t+1}(\xi_{t+1}^i)}w_{t+1}^i.
\]
Sufficient conditions for the weak convergence of such discretizations are discussed in~\cite{kar75}.

\subsection{Stochastic dual dynamic programming algorithm}\label{subsec:sddp}

Stochastic dual dynamic programming (SDDP) algorithm is a computational technique for solution of Bellman equations in convex stochastic optimization problems governed by finite-state Markov chains. The algorithm is based on successive approximation of the cost-to-go functions $(J_t)_{t=0}^T$ by polyhedral lower approximations. The SDDP algorithm was introduced
in~\cite{Pereira-Pinto:1991} for stagewise independent processes $\xi$ and later extended to finite-state Markov processes $\xi$ in~\cite{im96}; see also~\cite{pm12,lohndorf2019modeling} and the references there. The following version of the SDDP algorithm from \cite{pp25a} applies to control problems of form~\eqref{oc}. Note that, when the underlying process $\xi$ is a finite-state Markov chain, the conditional expectations below are finite sums weighted by the transition probabilities.

{\footnotesize
  \begin{enumerate}
  \item[0.]
    {\bf Initialization:} Choose convex (polyhedral) lower-approximations $J_t^0$ of the cost-to-go functions $J_t$ and set $k=0$.
  \item
    {\bf Forward pass:} Sample a path $\xi^k$ of $\xi$ and define $X^k$ by $X_0^k=x_0$ and
    \begin{align*}
U_t^{bk}&\in\argmin_{U^b_t\in\reals^{M^b}}\EE\left[\inf_{U^a_{t+1}\in\reals^{M^a}}\{L_t(X^k_t,U^b_t,U^a_{t+1}) + J^k_{t+1}(F_{t}(X^k_t,U^b_t,U^a_{t+1}))\}
                \Bigg{|}\xi^k_t\right],\\
      U_{t+1}^{ak} &\in\argmin_{U^a_{t+1}\in\reals^{M^a}}\{L_t(X^k_t,U^{bk}_t,U^a_{t+1},\xi^k_t,\xi^k_{t+1}) + J^k_{t+1}(F_{t}(X^k_t,U^{bk}_t,U^a_{t+1},\xi^k_t,\xi^k_{t+1}),\xi^k_{t+1})\},\\
      X^k_{t+1} &= F_{t}(X^k_t,U^{bk}_t,U^{ak}_{t+1},\xi^k_t,\xi^k_{t+1}).
    \end{align*}
  \item
    {\bf Backward pass:} Let $J_T^{k+1}:=J$ and, for $t=T-1,\ldots,0$,
    \begin{align*}
      \tilde J_t^{k+1}(X_t^k,\xi_t^k)
      &:= \inf_{U^b_t\in\reals^{M^b}}\EE\left[\inf_{U^a_{t+1}\in\reals^{M^a}}\{L_t(X^k_t,U^b_t,U^a_{t+1})
        + J^{k+1}_{t+1}(F_{t}(X^k_t,U^b_t,U^a_{t+1}))\}\Bigg{|}\xi^k_t\right],
      \\  
      V^{k+1}_t&\in\partial\tilde J^{k+1}_t(X_t^k,\xi_t^k)
    \end{align*}
    and let
    \[
      J_t^{k+1}(x_t,\xi_t^k) := \max\{J_t^k(x_t,\xi_t^k),\tilde J_t^{k+1}(X_t^k,\xi_t^k) + V_t^{k+1}\cdot(x_t-X^k_t)\}\quad\forall x_t\in\reals^N.
    \]
    Set $k:=k+1$ and go to 1.
  \end{enumerate}
}
Above, the dot ``$\cdot$'' stands for the Euclidean inner product. The minimization problems in the above algorithm are two-stage stochastic optimization problems while the original problem~\eqref{oc} has $T$ decision stages. Note also that the third minimization in the forward pass is one of the second-stage problems already solved in the second minimization. The functions $J_t^k$ generated by the algorithm are polyhedral lower approximations of the true cost-to-go functions $J_t$. It follows that the minimization
problems in the above recursions are finite-dimensional optimization problems that can be solved with standard convex optimization solvers. If the functions $L_t$ and $J$ are polyhedral, the problems are LP problems.

In the decision-observation format, where only the control $U^b$ is present, the two-stage problems reduce to a single-stage stochastic optimization problem, while in the observation-decision format without the $U^b$ control, only
deterministic single-stage problem need to be solved. This version of the algorithm is implemented e.g.\ in the Julia package described in~\cite{dowson2021sddp}.

\section{Computational results}\label{sec:comp}

In \S\ref{subsec:Data}, we provide numerical data for the model presented
in Section~\ref{sec:example}. In \S\ref{subsec:conv}, we present computational results obtained with the SDDP algorithm for the storage management problem. In particular, we study the accuracy and computational times obtained when the number $N$ of quadrature points is increased. \S\ref{Sensitivity_analysis} builds on the numerical optimization and calculates the indifference prices of different kinds of energy storages.

\subsection{Data}
\label{subsec:Data}

In the numerical studies below, we consider the problem of battery management over one day in the intraday market. We make hourly decisions, so the number of decision stages is $T=24$. We set the leakage rate to $l=0$ as there is practically no leakage in a 24-hour period considered here. Furthermore, we set the discharging efficiency to $c^-=1.05$, the charging efficiency to $c^+=0.95$, and the bid-ask spread coefficient to $\delta=1$€. These values are not derived from empirical data, but can be easily adjusted without affecting the computations. Throughout this section, we assume that the maximum charging and discharging speeds are equal, and we express them as a fraction $\alpha$ of the battery capacity. In other words, $\overline U=\underline U=\alpha\overline X^e$.

We discretize the one-dimensional
Markov process $\xi$ from \S\ref{subsec:price_model} using the conditional quadratures described in \S\ref{subsec:dmp}. We choose the sampling densities $\phi_t$ to be Gaussian with zero mean and variance $\sigma^2$. The corresponding measures are then approximated by Gauss-Hermite quadrature. For simplicity, we use the same number~$N$ of quadrature points in each stage $t$. We solve the resulting discretized Bellman equations with the Julia package {\tt SDDP.jl}~\cite{dowson2021sddp} which implements the Markov chain SDDP algorithm outlined in \S\ref{subsec:sddp}. Figure~\ref{fig:markov-chain} illustrates the discretization of the log-price over~15 days with $N=3$ quadrature points per time.

\begin{figure}[!ht]
  \centering  \includegraphics[width=.7\textwidth]{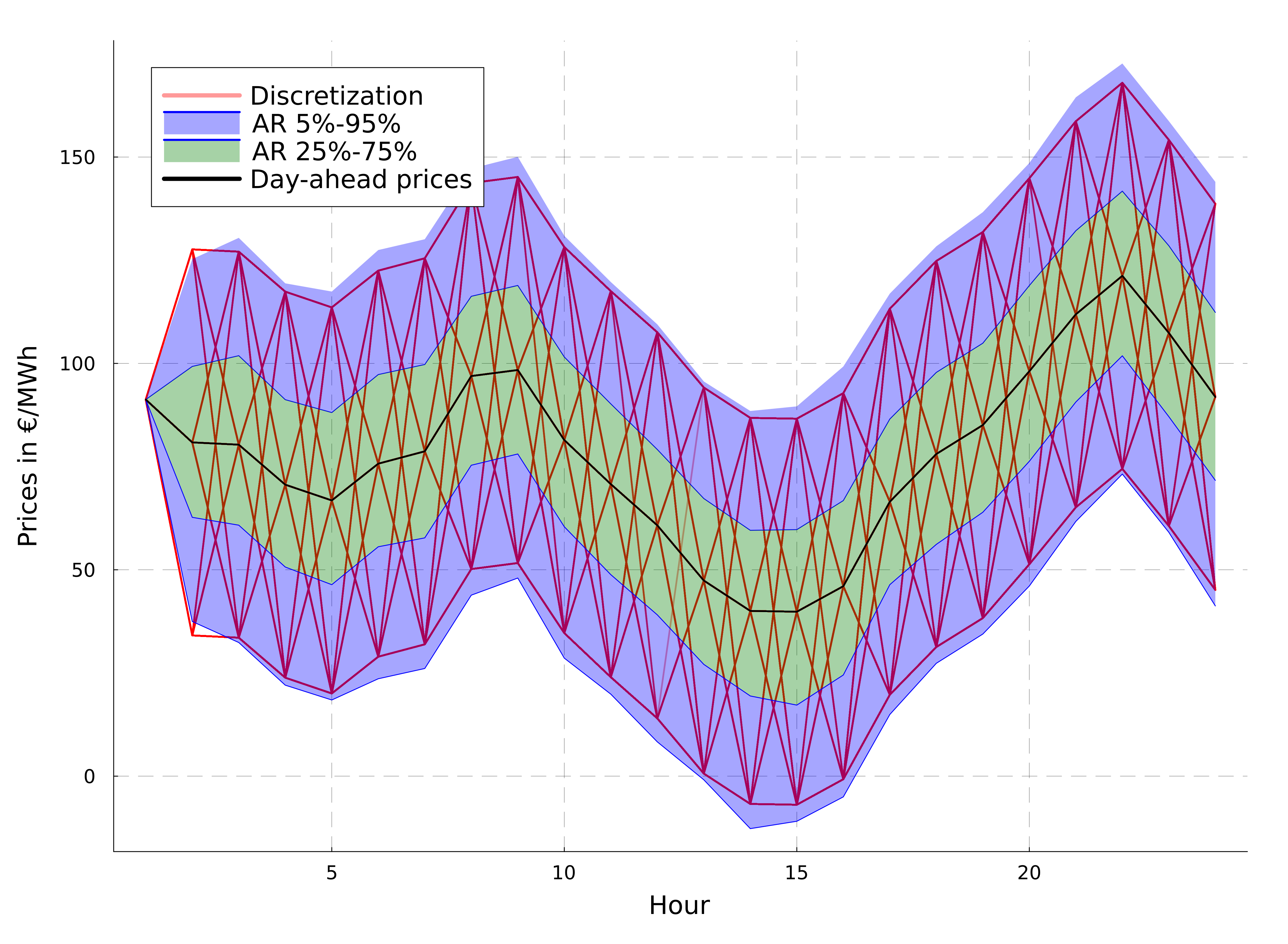}
  \caption{AR-1 process and its discretization as a Markov chain with 3 quadrature points per period \label{fig:markov-chain}}
\end{figure}


\subsection{Optimization of storage operation}\label{subsec:conv}

We study the performance of SDDP as we increase the number $N$ of quadrature points per period. For each $N$, we run 1000 iterations of SDDP. The risk aversion parameter $\rho$ is set to $0.03$ and the capacity to $\bar X_e=1$MWh. In the convergence analysis of this subsection, we will assume that $\alpha:=\overline U/\overline X^e=0.4$.

Table~\ref{tab:computation_times} gives the computation times for 1000 iterations of the SDDP algorithm for increasing the number $N$ of quadrature points per stage. The computation time increases roughly linearly with~$N$.

\begin{table}[!ht]
  \centering
  \begin{tabular}{lrrrrrrr}
    \hline
    $N$ & 1 & 2 & 4 & 8 & 16 & 32 & 64 \\
    \hline
    Time (s) & 201 & 44 & 77 & 144 & 270 & 820 & 1820  \\
    \hline
  \end{tabular}
  \caption{Time to compute 1000 iterations
    of SDDP as we increase the number~$N$ of quadrature points in the discretization of the AR process. \label{tab:computation_times}}
\end{table}

Each backward pass of the SDDP algorithm gives an upper bound for the optimum value in terms of the piecewise linear approximations of the Bellman functions. The convergence of the upper bounds is depicted in
Figure~\ref{fig:conv-sddp} for varying number $N$ of grid points per stage. In the deterministic case where $N=1$, the algorithm converges in less than 10 iterations. The convergence becomes slower when $N$ increases, but for $N\ge 4$, the upper bounds almost coincide.

\begin{figure}[!ht]
  \centering  \includegraphics[height=.5\textwidth,width=0.7\textwidth]{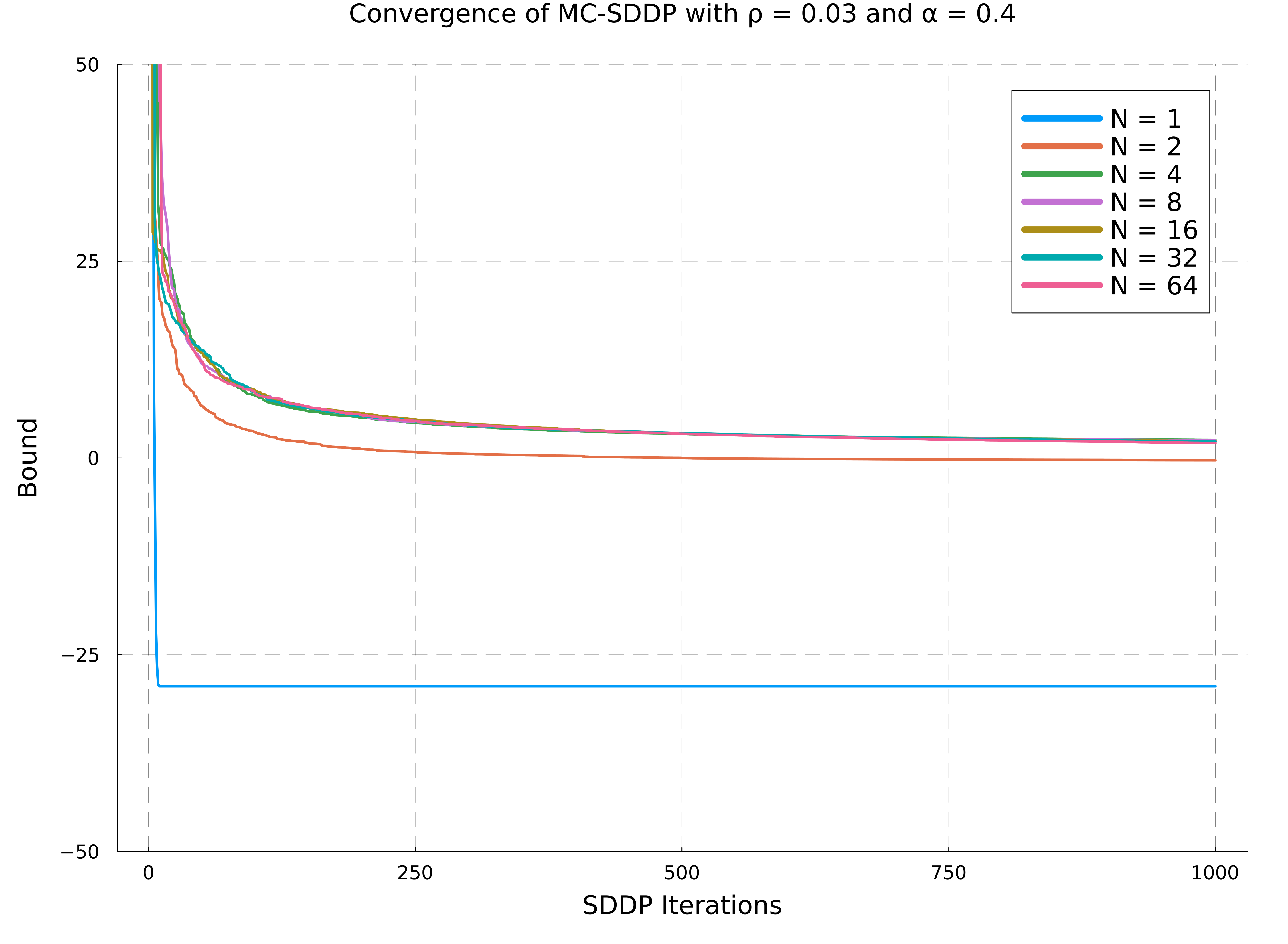}
  \caption{Convergence of SDDP as we increase the number of quadrature points in the discretization.
  \label{fig:conv-sddp}}
\end{figure}


In order to evaluate the quality of the cost-to-go functions provided by the
SDDP algorithm, we perform an out-of-sample evaluation along 10000 scenarios
randomly drawn from the continuously distributed Markov process described in
\S\ref{subsec:price_model}. In each scenario, we implement the control
obtained as in the forward pass of the SDDP algorithm as described in
\S\ref{subsec:sddp}. However, since the random out-of-sample values of
$\xi_t$ will not coincide with the Markov states $s_t^i$ in the finite-state
approximation, we use the control obtained in the state $s_t^i$ nearest to the
sampled value of $\xi_t$.

Figure~\ref{fig:assess-lb} plots the out-of-sample estimates of the objective
values together with the upper bounds found after 1000 iterations of the SDDP
algorithm. The figure also plots "in-sample" averages obtained by evaluating
the objective along scenarios randomly drawn from the finite-state Markov chain
approximation of the continuously distributed Markov process $\xi$.  The
out-of-sample averages increase and tend towards the SDDP upper bound as we
increase $N$. This is expected, as the Markov discretization is an approximation of the reference AR-1 process. However, for $N>8$, the improvements seem marginal.

\begin{figure}[!ht]
  \centering  \includegraphics[width=0.7\textwidth,height=0.45\textwidth]{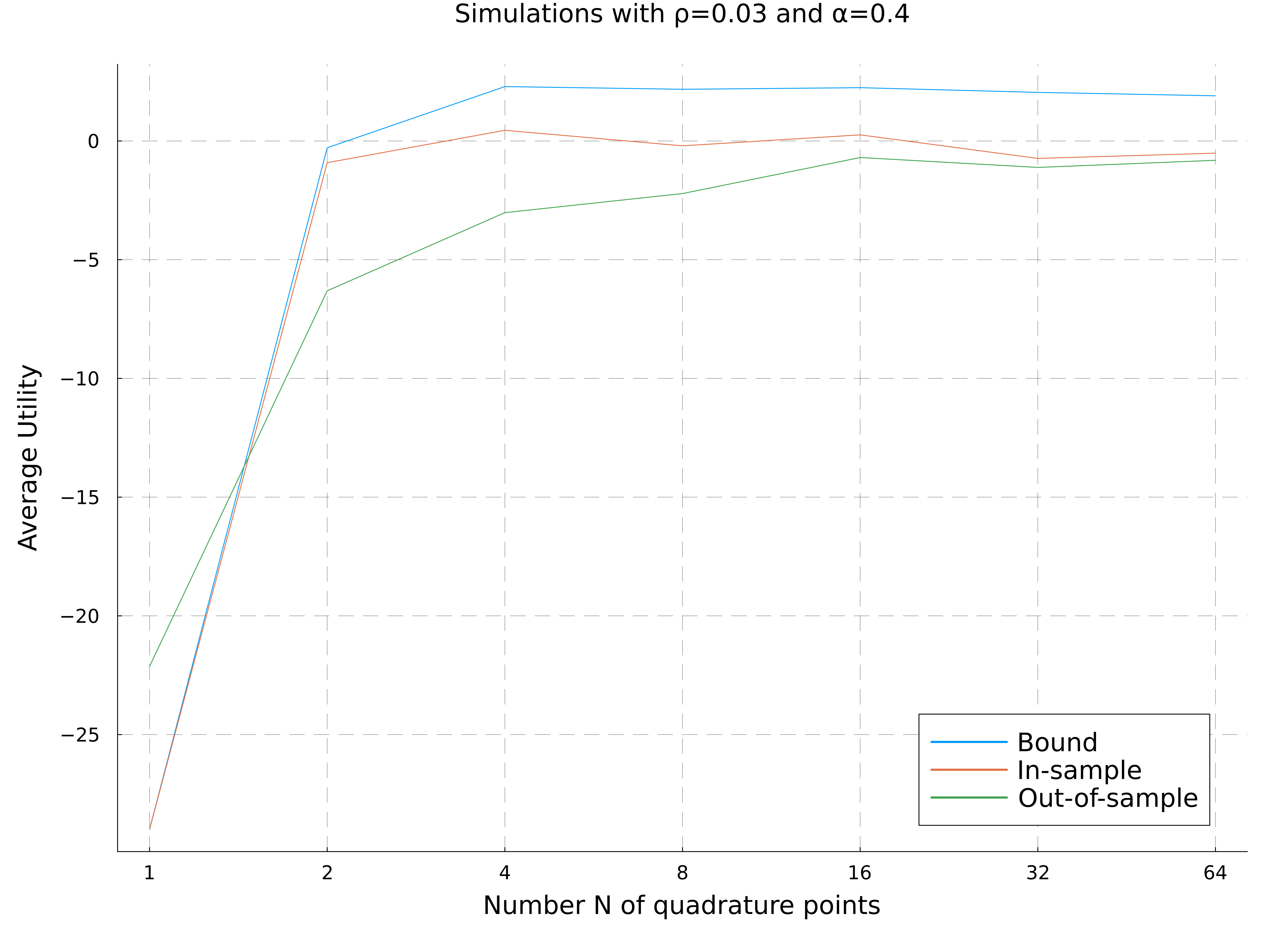}
  \caption{
    Optimum value estimates for increasing number of quadrature points. The red and blue values are given by SDDP while the green
    values are obtained with out of sample simulation using the optimized Bellman functions.
    \label{fig:assess-lb}}
\end{figure}




\subsection{Valuation of energy storages}
\label{Sensitivity_analysis}

As observed in \S\ref{subsec:conv}, the SDDP algorithm gives fairly accurate results already with $N = 8$ quadrature points per period. From now on, we fix $N=8$ and proceed to compute indifference prices of storage using the numerical evaluations of the optimum value function of~\eqref{sps} as described in
\S\ref{subsec:idp}. Since we are working with the exponential utility function, the indifference prices are given by Proposition~\ref{prop:idp} after the solution of just one instance of the battery management problem. This section studies the dependence of the optimal solutions and the indifference prices when we vary the parameters of the battery and the trader's risk aversion.

Figure~\ref{fig:sensitivity:risk} displays kernel-density plots of the terminal wealth obtained with strategies optimized for different risk aversion. As expected, higher risk aversion results in a thinner left tail of the terminal wealth distribution. The plots the kernel density estimates for two values of the charging speed. The effect of the risk aversion seems to influence the terminal wealth more for batteries with higher charging speed. A highly risk-averse agent will tend to operate a battery more sparingly due to the market price risk. A less risk-averse agent, on the other hand, tends to take full advantage of a high-speed battery. 

\begin{figure}[!ht]
  \centering
  \begin{tabular}{c}
  \includegraphics[height=.3\textwidth,width=0.7\textwidth]{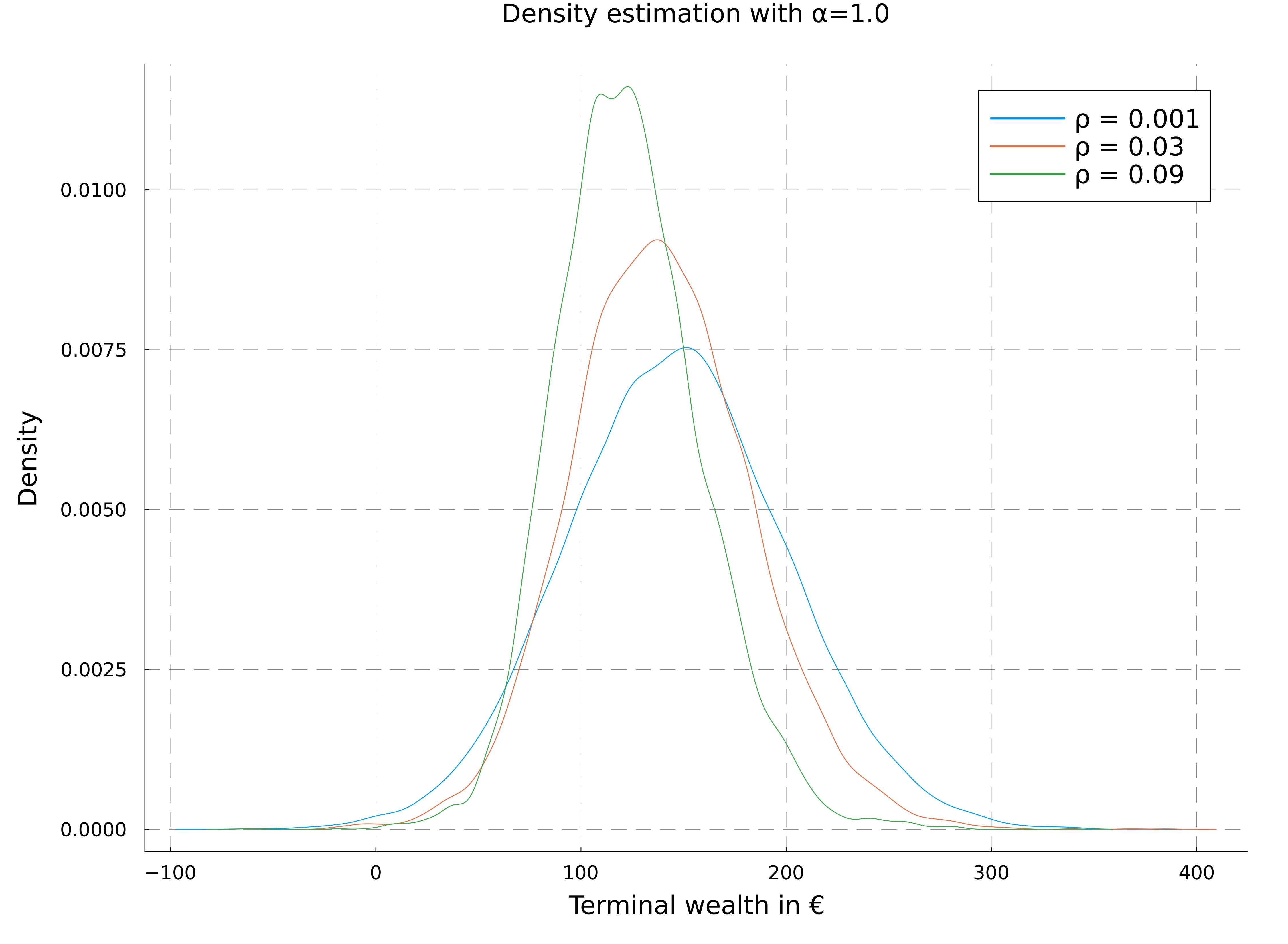} \\
  \includegraphics[height=.3\textwidth,width=0.7\textwidth]{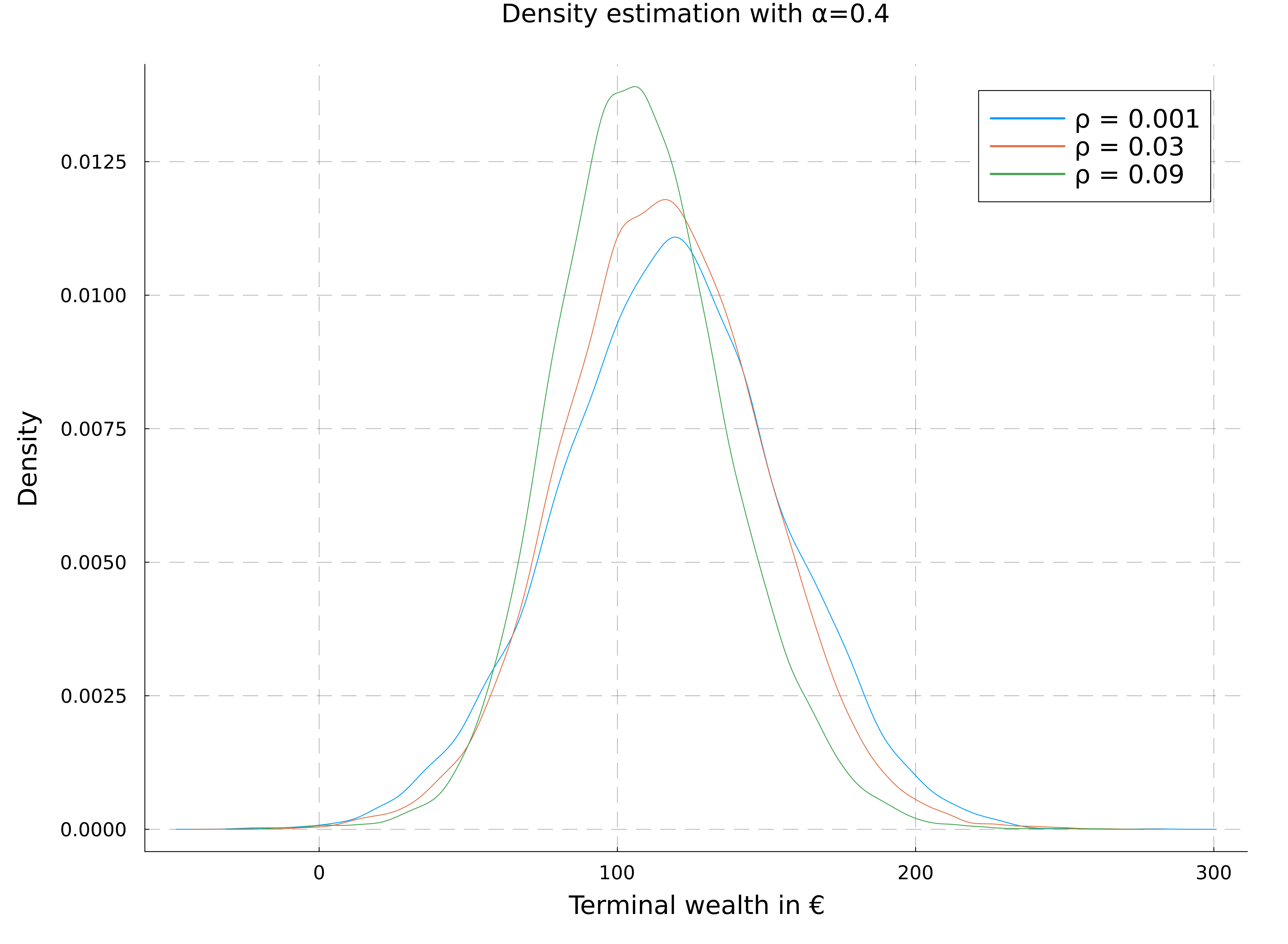}
  \end{tabular}
  \caption{
    Kernel density plots of the terminal wealth in out-of-sample simulations using 10000 scenarios for different values of the risk aversion parameter $\rho$ in the exponential utility function~\eqref{eq:exp}. The charging speed is set to $\alpha=0.4$ for the top figure, and to $\alpha=1$ for the bottom figure. The capacity is set to $\bar X_e=1$MWh. \label{fig:sensitivity:risk}
  }
\end{figure}

Figure~\ref{fig:indifference:capacity} plots indifference prices for
batteries with different capacities and agents with different risk aversions. As one might expect, the valuation of a storage increases with the storage capacity. With higher risk aversion, however, the valuations seem to saturate for larger capacities. This is quite natural as a highly risk-averse agent will avoid large positions in stored electricity due to the price risk. Figure~\ref{fig:indifference:rate} plots the indifference prices for storages with different charging speeds. Again, the valuations are more sensitive for less risk averse agents. Finally,
Figure~\ref{fig:indifference:rate} gives indifference prices for varying values
of the volatility parameter $\sigma$ in the electricity price model~\eqref{eq:ar_1}. The volatility can be thought of as the agent's uncertainty concerning the evolution of electricity prices. For all levels of risk aversion, higher volatility makes the storage more valuable to the agent.

\begin{figure}[!ht]
  \centering
  \includegraphics[width=.7\textwidth]{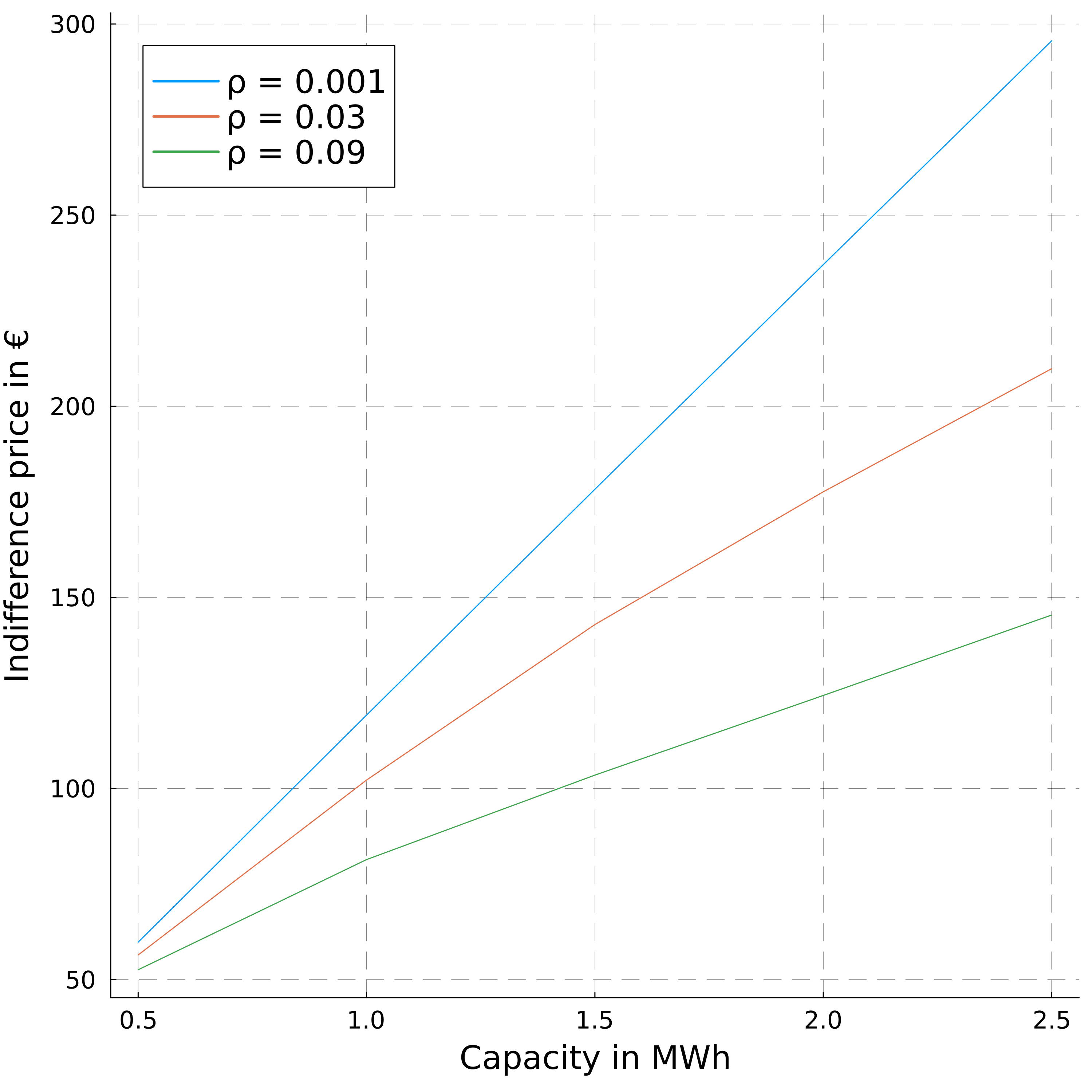}
  \caption{Indifference prices (in euros) of storages with different storage capacities.
    Different colors correspond to agents with different risk aversion parameters.
    Prices are given in euros and capacity in MWh. The charging speed is set to $\alpha=0.4$.
  \label{fig:indifference:capacity}}
\end{figure}

\begin{figure}[!ht]
  \centering
  \includegraphics[width=.7\textwidth]{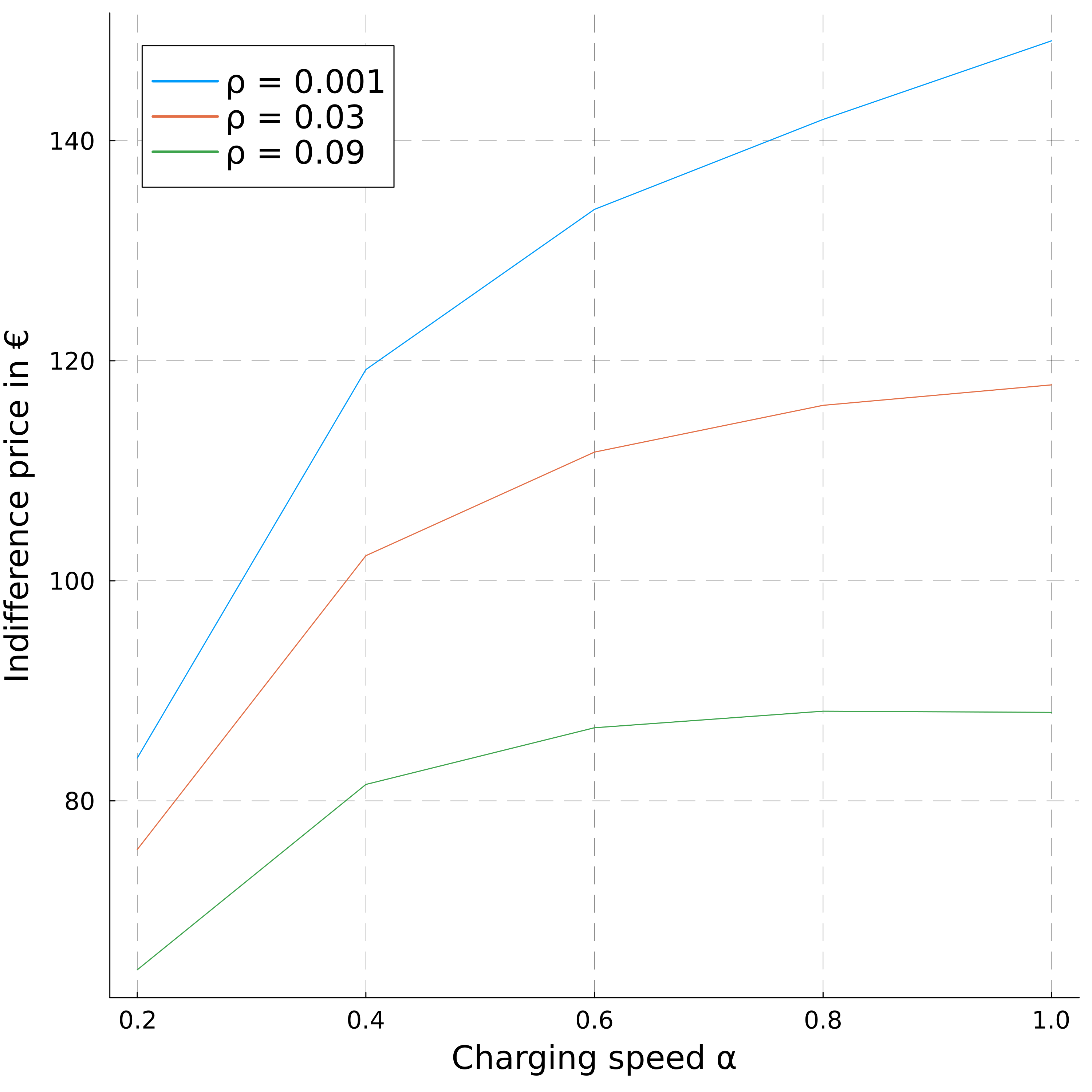}
  \caption{Indifference prices (in euros) of storages with different charging speeds.
    The horizontal axis gives the maximum hourly change in stored energy in units of the total storage capacity.
    Different colors correspond to agents with different risk aversion parameters. The capacity is set to $\bar X_m =1$MWh.
  \label{fig:indifference:rate}}
\end{figure}

\begin{figure}[!ht]
  \centering \includegraphics[width=.7\textwidth]{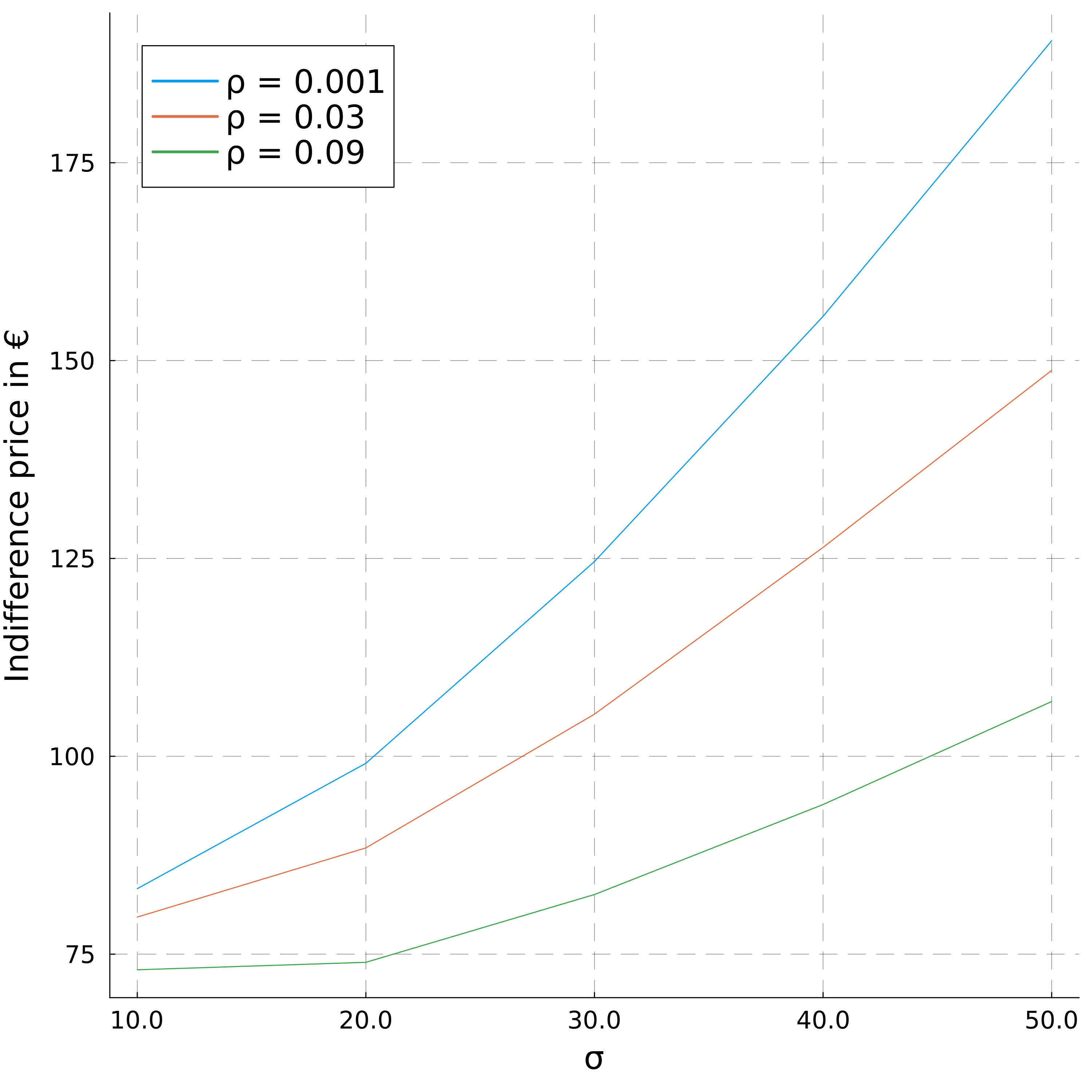}
  \caption{Indifference prices (in euros) obtained with different values of the volatility parameter $\sigma$ in the AR process~\eqref{eq:ar_1}. Different colors correspond to different risk aversions. The charging speed is set to $\alpha=0.4$ and the capacity to $\bar X_m =1$MWh.
  \label{fig:indifference:std}}
\end{figure}


\begin{thebibliography}{10}

\bibitem{arnold2011model}
M.~Arnold and G.~Andersson.
\newblock Model predictive control of energy storage including uncertain
  forecasts.
\newblock In {\em Power systems computation conference (PSCC), Stockholm,
  Sweden}, volume~23, pages 24--29. Citeseer, 2011.

\bibitem{ccgv}
G.~Callegaro, L.~Campi, V.~Giusto, and T.~Vargiolu.
\newblock Utility indifference pricing and hedging for structured contracts in
  energy markets.
\newblock {\em Math. Methods Oper. Res.}, 85(2):265--303, 2017.

\bibitem{car9}
R.~Carmona, editor.
\newblock {\em Indifference pricing: theory and applications}.
\newblock Princeton series in financial engineering. Princeton University
  Press, Princeton, NJ, 2009.

\bibitem{ccdmt}
P~Carpentier, J.-P. Chancelier, M.~De~Lara, T.~Martin, and T.~Rigaut.
\newblock Time block decomposition of multistage stochastic optimization
  problems.
\newblock {\em J. Convex Anal.}, 30(2):627--658, 2023.

\bibitem{dowson2021sddp}
O.~Dowson and L.~Kapelevich.
\newblock {SDDP.jl}: a {Julia} package for stochastic dual dynamic programming.
\newblock {\em INFORMS Journal on Computing}, 33(1):27--33, 2021.

\bibitem{gpw23}
M.~Germain, H.~Pham, and X.~Warin.
\newblock A level-set approach to the control of state-constrained
  {M}c{K}ean-{V}lasov equations: application to renewable energy storage and
  portfolio selection.
\newblock {\em Numer. Algebra Control Optim.}, 13(3-4):555--582, 2023.

\bibitem{im96}
G.~Infanger and D.~P. Morton.
\newblock Cut sharing for multistage stochastic linear programs with interstage
  dependency.
\newblock {\em Math. Programming}, 75(2):241--256, 1996.
\newblock Approximation and computation in stochastic programming.

\bibitem{kar75}
A.~F. Karr.
\newblock Weak convergence of a sequence of {M}arkov chains.
\newblock {\em Z. Wahrscheinlichkeitstheorie und Verw. Gebiete}, 33(1):41--48,
  1975/76.

\bibitem{lw21}
M.~L\"{o}hndorf and D.~Wozabal.
\newblock Gas storage valuation in incomplete markets.
\newblock {\em European J. Oper. Res.}, 288(1):318--330, 2021.

\bibitem{lohndorf2019modeling}
N.~L{\"o}hndorf and A.~Shapiro.
\newblock Modeling time-dependent randomness in stochastic dual dynamic
  programming.
\newblock {\em European Journal of Operational Research}, 273(2):650--661,
  2019.

\bibitem{lohndorf2013optimizing}
N.~L{\"o}hndorf, D.~Wozabal, and S.~Minner.
\newblock Optimizing trading decisions for hydro storage systems using
  approximate dual dynamic programming.
\newblock {\em Operations Research}, 61(4):810--823, 2013.

\bibitem{mp12}
P.~Malo and T.~Pennanen.
\newblock Reduced form modeling of limit order markets.
\newblock {\em Quantitative Finance}, 12(7):1025--1036, 2012.

\bibitem{megel2015stochastic}
O.~M{\'e}gel, J.~L Mathieu, and G.~Andersson.
\newblock Stochastic dual dynamic programming to schedule energy storage units
  providing multiple services.
\newblock In {\em 2015 IEEE Eindhoven PowerTech}, pages 1--6. IEEE, 2015.

\bibitem{pdccjr}
C.~Martinez Parra, M.~De~Lara, J.-P. Chancelier, P.~Carpentier, J.-M. Janin,
  and M.~Ruiz.
\newblock A two-timescale decision-hazard-decision formulation for storage
  usage values calculation.
\newblock {\em arXiv preprint arXiv:2408.17113}, 2024.

\bibitem{pen11}
T.~Pennanen.
\newblock Arbitrage and deflators in illiquid markets.
\newblock {\em Finance Stoch}, 15(1):57--83, 2011.

\bibitem{pen14}
T.~Pennanen.
\newblock Optimal investment and contingent claim valuation in illiquid
  markets.
\newblock {\em Finance Stoch}, 18(4):733--754, 2014.

\bibitem{pp24}
T.~Pennanen and A.-P. Perkki{\"o}.
\newblock {\em Convex Stochastic Optimization}.
\newblock Springer, 2024.

\bibitem{pp25a}
T.~Pennanen and A.-P. Perkki{\"o}.
\newblock Dynamic programming and dimensionality in convex stochastic control.
\newblock {\em Optimization Letters}, 2025.

\bibitem{Pereira-Pinto:1991}
M.~V.~F. Pereira and L.~M. V.~G. Pinto.
\newblock Multi-stage stochastic optimization applied to energy planning.
\newblock {\em Math. Program.}, 52:359--375, October 1991.

\bibitem{pm12}
A.B. Philpott and V.L. {de Matos}.
\newblock Dynamic sampling algorithms for multi-stage stochastic programs with
  risk aversion.
\newblock {\em European Journal of Operational Research}, 218(2):470--483,
  2012.

\bibitem{pv21}
A.~Picarelli and T.~Vargiolu.
\newblock Optimal management of pumped hydroelectric production with state
  constrained optimal control.
\newblock {\em J. Econom. Dynam. Control}, 126:Paper No. 103940, 21, 2021.

\bibitem{ptw9}
A.~Porchet, N.~Touzi, and X.~Warin.
\newblock Valuation of power plants by utility indifference and numerical
  computation.
\newblock {\em Math. Methods Oper. Res.}, 70(1):47--75, 2009.

\bibitem{roald2023power}
L.~A. Roald, D.~Pozo, A.~Papavasiliou, D.~K. Molzahn, J.~Kazempour, and
  A.~Conejo.
\newblock Power systems optimization under uncertainty: A review of methods and
  applications.
\newblock {\em Electric Power Systems Research}, 214:108725, 2023.

\bibitem{roc70a}
R.~T. Rockafellar.
\newblock {\em Convex analysis}.
\newblock Princeton Mathematical Series, No. 28. Princeton University Press,
  Princeton, N.J., 1970.

\bibitem{rw98}
R.~T. Rockafellar and R.~J.-B. Wets.
\newblock {\em Variational analysis}, volume 317 of {\em Grundlehren der
  Mathematischen Wissenschaften [Fundamental Principles of Mathematical
  Sciences]}.
\newblock Springer-Verlag, Berlin, 1998.

\bibitem{rte_imbalance_price}
RTE.
\newblock Imbalance settlement price, 2024.
\newblock
  \url{https://www.services-rte.com/en/view-data-published-by-rte/balancing.html}.

\bibitem{tauchen1991quadrature}
G.~Tauchen and R.~Hussey.
\newblock Quadrature-based methods for obtaining approximate solutions to
  nonlinear asset pricing models.
\newblock {\em Econometrica: Journal of the Econometric Society}, pages
  371--396, 1991.

\bibitem{weitzel2018energy}
T.~Weitzel and Ch.~H. Glock.
\newblock Energy management for stationary electric energy storage systems: A
  systematic literature review.
\newblock {\em European Journal of Operational Research}, 264(2):582--606,
  2018.

\bibitem{yu2017stochastic}
N.~Yu and B.~Foggo.
\newblock Stochastic valuation of energy storage in wholesale power markets.
\newblock {\em Energy Economics}, 64:177--185, 2017.

\bibitem{zheng2022arbitraging}
N.~Zheng, J.~Jaworski, and B.~Xu.
\newblock Arbitraging variable efficiency energy storage using analytical
  stochastic dynamic programming.
\newblock {\em IEEE Transactions on Power Systems}, 37(6):4785--4795, 2022.

\end{thebibliography}
\end{document}